\documentclass[numbook,envcountsame,smallcondensed,draft]{amsart}

\usepackage{cite,amsmath,amssymb,amsthm,latexsym,a4,graphicx,gastex}

\newtheorem{theorem}{Theorem}[section]

\newtheorem{proposition}[theorem]{Proposition}

\newtheorem{lemma}[theorem]{Lemma}

\newtheorem{corollary}[theorem]{Corollary}

\newtheorem{remark}[theorem]{Remark}

\theoremstyle{definition}

\newtheorem{problem}[theorem]{Problem}

\DeclareMathOperator{\con}{con}

\DeclareMathOperator{\los}{los}

\DeclareMathOperator{\var}{var}

\numberwithin{equation}{section}

\makeatletter

\renewcommand*\subjclass[2][2010]{\def\@subjclass{#2}\@ifundefined{subjclassname@#1}{\ClassWarning{\@classname}{Unknown edition (#1) of Mathematics Subject Classification; using '2010'.}}{\@xp\let\@xp\subjclassname\csname subjclassname@#1\endcsname}}

\makeatother

\begin{document}

\title{The lattice of varieties of implication semigroups}
\thanks{The first and the third author were partially supported by the Ministry of Education and Science of the Russian Federation (project 1.6018.2017/8.9) and by the Russian Foundation for Basic Research (grant No.\,17-01-00551).}

\author[S. V. Gusev]{Sergey V. Gusev}
\address{S.V.~Gusev, B.M.~Vernikov: Institute of Natural Sciences and Mathematics, Ural Federal University, Lenina str. 51, 620000 Ekaterinburg, Russia}
\email{sergey.gusb@gmail.com, bvernikov@gmail.com}

\author[H. P. Sankappanavar]{Hanamantagouda P. Sankappanavar}
\address{H.P.~Sankappanavar: Department of Mathematics, State University of New York, New Paltz, NY 12561, New York, U.S.A.}
\email{sankapph@newpaltz.edu}

\author[B. M. Vernikov]{Boris M. Vernikov}



\begin{abstract}
An implication semigroup is an algebra of type $(2,0)$ with a binary operation $\rightarrow$ and a \mbox{0-ary} operation 0 satisfying the identities $(x\rightarrow y)\rightarrow z\approx x\rightarrow(y\rightarrow z)$, $(x\rightarrow y)\rightarrow z\approx\left[(z'\rightarrow x)\rightarrow(y\rightarrow z)'\right]'$ and $0''\approx 0$ where $\mathbf u'$ means $\mathbf u\rightarrow 0$ for any term $\mathbf u$. We completely describe the lattice of varieties of implication semigroups. It turns out that this lattice is non-modular and consists of 16 elements.
\end{abstract}

\maketitle

\section{Introduction and summary}
\label{introduction}

In the article \cite{Sankappanavar-12}, the second author introduced and examined a new type of algebras as a generalization of De Morgan algebras. These algebras are of type $(2,0)$ with a binary operation $\rightarrow$ and a \mbox{0-ary} operation 0 satisfying the identities
\begin{align}
\label{I}&(x\to y)\to z\approx((z'\to x)\to(y\to z)')', \\
\label{0''=0}&0''\approx 0
\end{align}
where $\mathbf u'$ means $\mathbf u\rightarrow 0$ for any term $\mathbf u$. Such algebras are called \emph{implication zroupoids}. We refer an interested reader to \cite{Sankappanavar-12} for detailed explanation of the background and motivations.

The class of all implication zroupoids is a variety denoted by $\mathbf{IZ}$. It seems very natural to examine the lattice of its subvarieties. One of the important and interesting subvarieties of $\mathbf{IZ}$ is the class of all associative implication zroupoids, that is algebras from $\mathbf{IZ}$ satisfying the identity
$$
(x\rightarrow y)\rightarrow z\approx x\rightarrow(y\rightarrow z).
$$
It is natural to call such algebras \emph{implication semigroups}. The class $\mathbf{IS}$ of all implication semigroups forms a subvariety in $\mathbf{IZ}$. This subvariety was implicitly mentioned in \cite[Lemma 8.21]{Sankappanavar-12} and investigated more explicitly in the articles \cite {Cornejo-Sankappanavar-17a,Cornejo-Sankappanavar-17b,Cornejo-Sankappanavar-18}. (Incidentally, we should mention here that implication zroupoids are referred to as ``implicator groupoids'' in \cite{Cornejo-Sankappanavar-17b}.) But only the location of $\mathbf{IS}$ in the subvariety lattice of the variety $\mathbf{IZ}$ and ``interaction" of $\mathbf{IS}$ with other varieties from this lattice were studied in those articles. The aim of this paper is to examine the lattice of subvarieties of the variety $\mathbf{IS}$. Our main result gives a complete description of this lattice.

For convenience of our considerations, we turn to the notation generally accepted in the semigroup theory. As usual, we denote the binary operation by the absence of a symbol, rather than by $\rightarrow$. Since this operation is associative, we will, as a rule, omit brackets in terms. Besides that, the notation 0 for the \mbox{0-ary} operation seems to be inappropriate in the framework of examination of implication semigroups, because it is associated with the operation of fixing the zero element in a semigroup with zero. For this reason, we will denote the \mbox{0-ary} operation by the symbol $\omega$ which does not have any predefined a priori meaning. In this notation, implication semigroups are defined by the associative law $(xy)z\approx x(yz)$ and the following two identities:
\begin{align}
\label{xyz=z omega xyz omega omega}xyz&\approx z\omega xyz\omega^2,\\
\label{omega omega omega=omega}\omega^3&\approx\omega.
\end{align}

To formulate the main result of the article, we need some notation. As usual, elements of the free implication semigroup over a countably infinite alphabet are called \emph{words}, while elements of this alphabet are called \emph{letters}. Words rather than letters are written in bold. We connect two sides of identities by the symbol $\approx$. We denote by $\mathbf T$ the trivial variety of implication semigroups. The variety of implication semigroups given (within $\mathbf{IS}$) by the identity system $\Sigma$ is denoted by $\var\Sigma$. Let us fix notation for the following concrete varieties:
\begin{align*}
&\mathbf B:=\var\{x\approx x^2\},\\
&\mathbf K:=\var\{xyz\approx x^2\approx\omega,\,xy\approx yx\},\\
&\mathbf L:=\var\{xyz\approx x^2\approx\omega\},\\
&\mathbf M:=\var\{xyz\approx\omega,\,xy\approx yx\},\\
&\mathbf N:=\var\{xyz\approx\omega\},\\
&\mathbf{SL}:=\var\{x\approx x^2,\,xy\approx yx\},\\
&\mathbf{ZM}:=\var\{xy\approx\omega\}.
\end{align*}
The lattice of all varieties of implication semigroups is denoted by $\mathbb{IS}$.

The main result of the article is the following

\begin{theorem}
\label{main result}
The lattice $\mathbb{IS}$ has the form shown in Fig. \textup{\ref{IS}}.
\end{theorem}

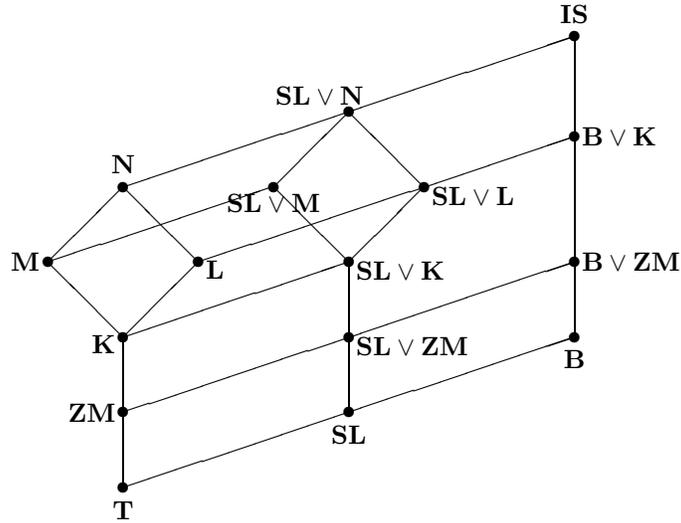
\begin{figure}[htb]
\unitlength=1mm
\linethickness{0.4pt}
\begin{center}
\begin{picture}(89,70)
\put(5,35){\circle*{1.33}}
\put(15,5){\circle*{1.33}}
\put(15,15){\circle*{1.33}}
\put(15,25){\circle*{1.33}}
\put(15,45){\circle*{1.33}}
\put(25,35){\circle*{1.33}}
\put(35,45){\circle*{1.33}}
\put(45,15){\circle*{1.33}}
\put(45,25){\circle*{1.33}}
\put(45,35){\circle*{1.33}}
\put(45,55){\circle*{1.33}}
\put(55,45){\circle*{1.33}}
\put(75,25){\circle*{1.33}}
\put(75,35){\circle*{1.33}}
\put(75,51.67){\circle*{1.33}}
\put(75,65){\circle*{1.33}}
\put(15,5){\line(0,1){20}}
\put(15,25){\line(-1,1){10}}
\put(15,25){\line(1,1){10}}
\put(5,35){\line(1,1){10}}
\put(25,35){\line(-1,1){10}}
\put(45,15){\line(0,1){20}}
\put(45,35){\line(-1,1){10}}
\put(45,35){\line(1,1){10}}
\put(35,45){\line(1,1){10}}
\put(55,45){\line(-1,1){10}}
\put(75,25){\line(0,1){40}}
\put(15,5){\line(3,1){60}}
\put(15,15){\line(3,1){60}}
\put(15,25){\line(3,1){30}}
\put(5,35){\line(3,1){30}}
\put(25,35){\line(3,1){50}}
\put(15,45){\line(3,1){60}}
\put(75,22){\makebox(0,0)[cc]{$\mathbf B$}}
\put(76,35){\makebox(0,0)[lc]{$\mathbf{B\vee ZM}$}}
\put(76,51.67){\makebox(0,0)[lc]{$\mathbf{B\vee K}$}}
\put(75,68){\makebox(0,0)[cc]{$\mathbf{IS}$}}
\put(14,24){\makebox(0,0)[rc]{$\mathbf K$}}
\put(26,34){\makebox(0,0)[lc]{$\mathbf L$}}
\put(2,35){\makebox(0,0)[cc]{$\mathbf M$}}
\put(15,48){\makebox(0,0)[cc]{$\mathbf N$}}
\put(45,12){\makebox(0,0)[cc]{$\mathbf{SL}$}}
\put(46,35){\makebox(0,0)[lt]{$\mathbf{SL\vee K}$}}
\put(56,45){\makebox(0,0)[lt]{$\mathbf{SL\vee L}$}}
\put(35,44){\makebox(0,0)[ct]{$\mathbf{SL\vee M}$}}
\put(47,56){\makebox(0,0)[rb]{$\mathbf{SL\vee N}$}}
\put(46,25){\makebox(0,0)[lt]{$\mathbf{SL\vee ZM}$}}
\put(15,2){\makebox(0,0)[cc]{$\mathbf T$}}
\put(14,15){\makebox(0,0)[rc]{$\mathbf{ZM}$}}
\end{picture}
\end{center}
\caption{The lattice $\mathbb{IS}$}
\label{IS}
\end{figure}

In \cite[Problem 5]{Sankappanavar-12}, the second author formulated the question of whether the lattice of all varieties of implication zroupoids is distributive. The following assertion immediately follows from Fig. \ref{IS} and provides the negative answer to this question.

\begin{corollary}
\label{non-modular}
The lattice $\mathbb{IS}$ is non-modular.\qed
\end{corollary}

This article consists of three sections and two appendixes. Section \ref{proof} is devoted to the proof of Theorem \ref{main result}, while Section \ref{problems} contains several open problems. The proof of Theorem \ref{main result} given in Section \ref{proof} is based in essential degree on suggestions of the referee on the journal version of the article. The original version of the proof due to the authors is set forth in Appendix \ref{original proof}. Finally, Appendix \ref{about 0-distr} contains some assertions concerning with one of the problems posed in Section \ref{problems}.

\section{Proof of the main result}
\label{proof}

To verify Theorem \ref{main result}, we need a few auxiliary assertions. If $\mathbf u$ and $\mathbf v$ are words and $\varepsilon$ is an identity then we will write $\mathbf{u\stackrel{\varepsilon}\approx v}$ in the case when the identity $\mathbf{u\approx v}$ follows from $\varepsilon$.

\begin{lemma}
\label{3 identities}
The variety $\mathbf{IS}$ satisfies the following identities:
\begin{align}
\label{omega omega=omega} \omega^2&\approx\omega,\\
\label{omega x=x omega} \omega x&\approx x\omega,\\
\label{xyz=xyz omega} xyz&\approx xyz\omega.
\end{align}
\end{lemma}

\begin{proof}
The following three chains of identities provide deductions of the identities \eqref{omega omega=omega}--\eqref{xyz=xyz omega}:
\begin{align*}
\text{\eqref{omega omega=omega}}:\quad&\omega^2\stackrel{\eqref{omega omega omega=omega}}\approx\omega^{10}=\omega^2\omega\omega\omega^2\omega^2\omega^2\stackrel{\eqref{xyz=z omega xyz omega omega}}\approx\omega\omega^2\omega^2=\omega^5\stackrel{\eqref{omega omega omega=omega}}\approx\omega,\\
\text{\eqref{omega x=x omega}}:\quad&\omega x\stackrel{\eqref{omega omega=omega}}\approx\omega\omega x\stackrel{\eqref{xyz=z omega xyz omega omega}}\approx x\omega\omega\omega x\omega^2\stackrel{\eqref{omega omega=omega}}\approx(x\omega\omega\omega x\omega^2)\omega\\
&\phantom{\omega x\,}\stackrel{\eqref{omega omega=omega}}\approx(\omega\omega x)\omega\approx\omega\omega x\omega\omega\omega^2\stackrel{\eqref{xyz=z omega xyz omega omega}}\approx x\omega\omega\stackrel{\eqref{omega omega=omega}}\approx x\omega,\\
\text{\eqref{xyz=xyz omega}}:\quad&xyz\stackrel{\eqref{xyz=z omega xyz omega omega}}\approx z\omega xyz\omega^2\stackrel{\eqref{omega omega=omega}}\approx(z\omega xyz\omega^2)\omega\stackrel{\eqref{xyz=z omega xyz omega omega}}\approx(xyz)\omega.
\end{align*}
Lemma is proved.
\end{proof}

An idempotent $e$ of a semigroup $S$ that commutes with every element in $S$ is said to be a \emph{central idempotent}. The identities \eqref{omega omega=omega} and \eqref{omega x=x omega} show that if $S$ is an implication semigroup then the distinguished element $\omega$ of $S$ is a central idempotent. This explains our interest in the following assertion which is a part of semigroup folklore. We provide its proof here for the sake of completeness.

\begin{lemma}
\label{S is a subdirect product} 
If $e$ is a central idempotent of a semigroup $S$ then $S$ is a subdirect product of its ideal $eS$ and the Rees quotient $S/eS$.
\end{lemma}

\begin{proof}
Clearly, $eS$ is an ideal of $S$ and the natural homomorphism $\eta \colon S \to S/eS$ has the property that $\eta(x)=\eta(y)$ if and only if either $x=y$ or $x, y \in eS$. On the other hand, the map $\varphi\colon S\to eS$ given by the rule $\varphi(x)=ex$ is a homomorphism of $S$ onto $eS$ and $ex=ey$ implies $x=y$ for $x, y \in eS$. Therefore, if $x, y \in S$ are such that $\eta(x) =\eta(y)$ and $\varphi(x)=\varphi(y)$, then $x=y$. We see that $\varphi$ and $\eta$ are surjective homomorphisms from $S$ onto $eS$ and $S/eS$ respectively, and the intersection of kernels of these homomorphisms is the equality relation.  Hence $S$ is a subdirect product of $eS$ and $S/eS$.
\end{proof}

Recall that a semigroup is called a \emph{band} if it satisfies the identity $x^2\approx x$. We call a variety of implication semigroups $\mathbf V$ a \emph{monoid variety} if the identities $x\omega\approx\omega x\approx x$ hold in $\mathbf V$. Obviously, this means that every semigroup in $\mathbf V$ has an identity element and the operation $\omega$ fixes just this element in each semigroup from $\mathbf V$.

\begin{lemma}
\label{bands=monoids}
A variety of implication semigroups is a monoid variety if and only if it is a variety of bands.
\end{lemma}

\begin{proof}
Any monoid variety satisfies the identities $x\approx \omega^2 x\stackrel{\eqref{xyz=z omega xyz omega omega}}\approx x\omega^3x\omega^2\approx x^2$, while any variety of bands satisfies the identities $\omega x\stackrel{\eqref{omega x=x omega}}\approx x\omega\approx x^3\omega\stackrel{\eqref{xyz=xyz omega}}\approx x^3\approx x$.
\end{proof}

\begin{lemma}
\label{V=(V wedge B) vee (V wedge N)}
If $\mathbf V$ is an implication semigroup variety then $\mathbf V= (\mathbf V \wedge \mathbf B) \vee (\mathbf V \wedge \mathbf N)$.
\end{lemma} 

\begin{proof}
We can assume that the variety $\mathbf V$ is generated by an implication semigroup $S$. In view of Lemmas \ref{3 identities} and \ref{S is a subdirect product}, the set $\omega S$ is an ideal of $S$ and $S$ is a subdirect product of $\omega S$ and the Rees quotient $S/\omega S$. Clearly, $\omega S$ is an implication semigroup with the distinguished element $\omega$ and $\omega x=x\omega=x$ for every $x\in\omega S$. Then $\omega S\in\mathbf B$ by Lemma \ref{bands=monoids}. Note also that $S/\omega S$ is an implication semigroup with the distinguished element $\omega S$ and $xyz\stackrel{\eqref{xyz=xyz omega}}=xyz\omega\stackrel{\eqref{omega x=x omega}}=\omega xyz\in \omega S$ for every $x,y,z\in S$. This implies that $S/\omega S$ satisfies the identity $xyz\approx \omega$ and therefore, is contained in the variety $\mathbf N$. Thus, we have proved that $S$ is a subdirect product of the implication semigroups $\omega S\in\mathbf B$ and $S/\omega S\in\mathbf N$. This implies the required conclusion.
\end{proof}

As usual, we denote by $L({\mathbf X})$ the subvariety lattice of the variety $\mathbf X$. 

\begin{proof}[Proof of Theorem \textup{\ref{main result}}]
According to Lemma \ref{bands=monoids}, $\mathbf B$ is a monoid variety. Therefore, it satisfies the identities $xyx\approx x\omega yx\omega^2 \stackrel{\eqref{xyz=z omega xyz omega omega}}\approx \omega yx\approx yx$. The lattice of varieties of band monoids is completely described in \cite{Wismath-86}. In view of \cite[Proposition 4.7]{Wismath-86}, the lattice $L(\mathbf B)$ is the 3-element chain $\mathbf{T\subset SL\subset B}$.

The variety $\mathbf N$ satisfies the identities $\omega x\stackrel{\eqref{omega x=x omega}}\approx x\omega\stackrel{\eqref{omega omega=omega}}\approx x\omega^2\approx\omega$. Hence every semigroup from $\mathbf N$ contains the zero element and the operation $\omega$ fixes just this element in each semigroup from $\mathbf N$. This means that $\mathbf N$ is nothing but the variety of all 3-nilpotent semigroups. The subvariety lattice of this variety has the form shown in Fig. \ref{IS}. This claim can be easily verified directly and is a part of semigroup folklore. It is known at least from the beginning of 1970's (see \cite{Melnik-72}, for instance).

Recall that commutative bands are called \emph{semilattices}. We fix notation for the following semigroups:
\begin{align*}
A&:=\{0,1\}\text{ --- the 2-element semilattice},\\ 
B&:=\langle e,f,1\mid ef=f^2=f, fe=e^2=e\rangle=\{e,f,1\},\\
K&:=\langle a, b,0 \mid ab=ba, a^2=b^2=0 \rangle=\{a,b,ab,0\}\\
L&:=\langle a, b,0 \mid ba=a^2=b^2=0 \rangle=\{a,b,ab,0\},\\
M&:=\langle a, b,0 \mid ab=ba, a^2=ab^2=b^3=0 \rangle=\{a,b,b^2,ab,0\},\\
Z&:=\langle a,0\mid a^2=0\rangle=\{a,0\}
\end{align*}
where 0 and 1 have the usual sense in semigroup context (the zero element of a semigroup and the identity one, respectively). All these semigroups can be considered as implication semigroups. Indeed, it is easy to see that putting $\omega=1$ in $A$, $B$ and $\omega=0$ in $K$, $L$, $M$, $Z$, we achieve the fulfillment of the identities \eqref{xyz=z omega xyz omega omega} and \eqref{omega omega omega=omega}. The variety generated by an implication semigroup $S$ is denoted by $\var S$. It is well known and easily verified that $\mathbf B=\var B$, $\mathbf K=\var K$, $\mathbf L=\var L$, $\mathbf M=\var M$, $\mathbf{SL}=\var A$ and $\mathbf{ZM}=\var Z$.

Now we are going to prove that the lattice $L(\mathbf{SL\vee N})$ has the form shown in Fig. \ref{IS}. Clearly, the implication semigroups $A$, $L$ and $M$ satisfy the identity $xy\omega\approx yx\omega$. So, this identity holds in $\mathbf{SL\vee N}$. Since it is false in $B$, we have that $(\mathbf{SL\vee N})\wedge\mathbf B=\mathbf{SL}$. This fact and Lemma \ref{V=(V wedge B) vee (V wedge N)} imply that $\mathbf V= (\mathbf V \wedge \mathbf{SL}) \vee (\mathbf V \wedge \mathbf N)$ for every subvariety $\mathbf V$ of $\mathbf{SL\vee N}$. Then $\mathbf{SL\vee N}$ has at most 12 subvarieties, namely, the ones shown in Fig. \ref{IS}. We need to verify that these subvarieties are different from each other. For a class $\mathbf X$ of implication semigroups, let $\overline{\mathbf X}$ stand for the class of all semigroup reducts of implication semigroups in $\mathbf X$. Since $\omega\approx x^3$  in $\mathbf N$, we see that $\overline{\mathbf V}$ is a subvariety of $\overline{\mathbf N}$ whenever $\mathbf V$ is a subvariety of $\mathbf N$. Now let $\mathbf V$ and $\mathbf W$ be two different subvarieties of $\mathbf N$. Then the semigroup varieties $\overline{\mathbf V}$ and $\overline{\mathbf W}$ are different as well. It is well known that the semigroup variety $\overline{\mathbf{SL}}$ of all semilattices constitutes a neutral element of the lattice of all semigroup varieties (it is proved explicitly in \cite[Proposition 4.1]{Volkov-05}), whence $\overline{\mathbf V}\vee\overline{\mathbf{SL}}\ne\overline{\mathbf W}\vee\overline{\mathbf{SL}}$. An identity $\mathbf{u\approx v}$ is called a \emph{semigroup identity} if both the words $\mathbf u$ and $\mathbf v$ do not contain the symbol of \mbox{0-ary} operation. Any semigroup identity that differentiates $\overline{\mathbf V}\vee\overline{\mathbf{SL}}$ from $\overline{\mathbf W}\vee\overline{\mathbf{SL}}$ will also differentiate the implication semigroup varieties $\mathbf{V\vee SL}$ and $\mathbf{W\vee SL}$.

Further, we are going to prove that the lattice $L(\mathbf{B\vee ZM})$ has the form shown in Fig. \ref{IS}. First of all, we note that the identity $xy\approx xy\omega$ holds in $B$ and $Z$ but fails in $K$. Therefore, $(\mathbf{B\vee ZM})\wedge\mathbf N=\mathbf{ZM}$. This fact and Lemma \ref{V=(V wedge B) vee (V wedge N)} imply that $\mathbf V= (\mathbf V \wedge \mathbf B) \vee (\mathbf V \wedge \mathbf{ZM})$ for every subvariety $\mathbf V$ of $\mathbf{B\vee ZM}$. Then $\mathbf{B\vee ZM}$ has at most 6 subvarieties, namely, the ones shown in Fig. \ref{IS}. We need to verify that these subvarieties are different from each other. In view of the observations made in the first, the second and the fourth paragraphs of the proof of Theorem \ref{main result}, it remains to show that $\mathbf{SL\vee ZM}\subset \mathbf{B\vee ZM}$. This follows from the fact that the identity $xy\approx yx\omega$ holds in $\mathbf{SL\vee ZM}$ but fails in $\mathbf B$. 

Lemma \ref{V=(V wedge B) vee (V wedge N)} with $\mathbf{V=IS}$ implies that $\mathbf{IS} =\mathbf B \vee \mathbf N$. Since $\mathbf B$ has exactly 3 subvarieties and $\mathbf N$ has exactly 6 ones, we have that $\mathbf{IS}$ has at most 18 subvarieties. Now we aim to show that $\mathbf{B\vee K= B\vee L}$ and $\mathbf{B\vee M=B\vee N}$. The subset $I=\{(e,0),(f,0),(1,0)\}$ of the direct product $B\times K$ forms an ideal of $B \times K$. The Rees quotient $(B\times K)/I$ is a 3-nilpotent implication semigroup that satisfies the identity $x^2 \approx \omega$ but violates the commutative law. Indeed, $(e, a)(f, b) = (f, ab)\ne (e, ab) = (f, b)(e, a)$. We see that $(B \times K)/I$ lies in $\mathbf L$ but does not lie in $\mathbf K$. Note that $\mathbf K$ is the only maximal subvariety of $\mathbf L$. Whence $(B \times K)/I$ generates the variety $\mathbf L$. Since $(B \times K)/I\in\mathbf B \vee\mathbf K$, we have that $\mathbf L\subseteq\mathbf B \vee\mathbf K$. We conclude that $\mathbf B \vee\mathbf L \subseteq\mathbf B \vee\mathbf K$, and the converse inclusion is clear. Thus $\mathbf{B\vee K=B\vee L}$. Further, $\mathbf{B\vee M\supseteq B\vee K=B\vee L}$. Therefore, $\mathbf{L\subseteq B\vee M}$, whence $\mathbf{N=M\vee L\subseteq B\vee M}$. We get that $\mathbf B\vee\mathbf N \subseteq \mathbf B\vee\mathbf M$. The converse inclusion is clear, whence $\mathbf{B\vee M=B\vee N}$. Thus, we have proved that $\mathbf{IS}$ has at most 16 subvarieties, namely, the ones shown in Fig. \ref{IS}. We need to verify that these subvarieties are different from each other. In view of what is said in the fourth and the fifth paragraphs of the proof of Theorem \ref{main result}, it remains to show that $\mathbf{B\vee K}\subset \mathbf{IS}$. This follows from the above-mentioned equalities $\mathbf{IS=B\vee N=B\vee M}$ and the fact that the identity $x\omega\approx x^2$ holds in $K$ and $B$ but fails in $M$.
\end{proof}

\section{Open problems}
\label{problems}

We denote by $\mathbb{IZ}$ the lattice of all varieties of  implication zroupoids. Theorem \ref{main result} shows that the lattice $\mathbb{IZ}$ is non-modular but the following problem still remains open.

\begin{problem}
\label{non-trivial identities?}
Determine whether the lattice $\mathbb{IZ}$ satisfies any non-trivial lattice identity.
\end{problem}

Recall that a lattice $\langle L;\vee,\wedge\rangle$ with the least element 0 is called 0-\emph{distribuive} if it satisfies the implication
$$
\forall x,y,z\in L:\quad x\wedge z=y\wedge z=0\longrightarrow(x\vee y)\wedge z=0.
$$
Lattices of varieties of all classical types of algebras (groups, semigroups, rings, lattices etc.) are well-known to be 0-distributive. The following question seems to be interesting.

\begin{problem}
\label{0-distributive?}
Determine whether the lattice $\mathbb{IZ}$ is $0$-distributive.
\end{problem}

This problem is closely related to knowing the set of all atoms of the lattice $\mathbb{IZ}$. This set is known but not yet published. Indeed, it is well known that any non-trivial variety of algebras contains a simple algebra, i.e. algebra without congruences except the trivial and the universal ones (see \cite[Theorem 10.13]{Burris-Sankappanavar-81}, for instance). The complete list of simple implication zroupoids is provided by \cite[Theorem 5.8]{Cornejo-Sankappanavar-16}. The variety generated by one of these algebras contains either $\mathbf{ZM}$ or $\mathbf{SL}$ or the variety $\mathbf{BA}$ of all Boolean algebras. On the other hand, it is easy to see that these three varieties are atoms of $\mathbb{IZ}$. Combining these observations, we have the following

\begin{remark}
\label{atoms}
The varieties $\mathbf{ZM}$, $\mathbf{SL}$ and $\mathbf{BA}$ are {the only} atoms of the lattice $\mathbb{IZ}$.
\end{remark}

Returning to Problem \ref{0-distributive?}, it is easy to see that this problem is equivalent to the following claim: if $\mathbf A$ is an atom of the lattice $\mathbb{IZ}$ and $\mathbf X$, $\mathbf Y$ are varieties of implication zroupoids with $\mathbf{X,Y\nsupseteq A}$ then $\mathbf{X\vee Y\nsupseteq A}$. We have a proof of this fact in the case when $\mathbf A$ is one of the varieties $\mathbf{SL}$ or $\mathbf{BA}$ (see Appendix \ref{about 0-distr}). But the case when $\mathbf{A=ZM}$ still remains open.

An element $x$ of a lattice $L$ is called \emph{neutral} if, for any $y,z\in L$, the elements $x$, $y$ and $z$ generate a distributive sublattice of $L$. Neutral elements play an important role in the lattice theory. If $a$ is a neutral element of a lattice $L$ then $L$ is a subdirect product of the principal ideal and the principal filter of $L$ generated by $a$ (see \cite[proof of Theorem 254]{Gratzer-11}). So, the knowledge of the set of neutral elements of a lattice gives significant and important information about the structure of this lattice. Fig. \ref{IS} shows that the varieties $\mathbf{SL}$ and $\mathbf{ZM}$ are neutral elements of the lattice $\mathbb{IS}$. The following problem seems to be very interesting.
 
\begin{problem}
\label{are atoms neutral?}
Determine whether $\mathbf{SL}$, $\mathbf{ZM}$ and $\mathbf{BA}$ are neutral elements of the lattice $\mathbb{IZ}$.
\end{problem}

Note that the varieties of all semilattices and of all semigroups with zero multiplication considered as simply semigroup varieties are neutral elements of the lattice of all semigroup varieties (see \cite[Proposition 4.1]{Volkov-05} or Theorem 3.4 in the survey \cite{Vernikov-15}).

\section*{APPENDIX A. An alternative proof of Theorem \ref{main result}}

\renewcommand{\thesection}{\Alph{section}}

\setcounter{section}{0}
\refstepcounter{section}
\label{original proof}

Here we give an alternative proof of Theorem \ref{main result}. This appendix is divided into 7 subsections.

\subsection{Preliminaries}
\label{preliminaries}

We need a number of auxiliary results.

\begin{lemma}
\label{2 identities}
The variety $\mathbf{IS}$ satisfies the following identities:
\begin{align}
\label{xyx=yx omega} xyx&\approx yx\omega,\\
\label{xxy=xy omega} x^2y&\approx xy\omega.
\end{align}
\end{lemma}

\begin{proof}
The following two chains of identities provide deductions of the identities \eqref{xyx=yx omega} and \eqref{xxy=xy omega}:
\begin{align*}
\text{\eqref{xyx=yx omega}}:\quad&xyx\stackrel{\eqref{xyz=xyz omega}}\approx xyx\omega^4\stackrel{\eqref{omega x=x omega}}\approx x\omega\omega yx\omega^2\stackrel{\eqref{xyz=z omega xyz omega omega}}\approx\omega yx\stackrel{\eqref{omega x=x omega}}\approx yx\omega,\\
\text{\eqref{xxy=xy omega}}:\quad&x^2y\stackrel{\eqref{xyz=xyz omega}}\approx x^2y\omega\stackrel{\eqref{omega x=x omega}}\approx x\omega xy\stackrel{\eqref{xyx=yx omega}}\approx\omega x\omega y\stackrel{\eqref{omega x=x omega}}\approx\omega^2 xy\stackrel{\eqref{omega omega=omega}}\approx\omega xy\stackrel{\eqref{omega x=x omega}}\approx xy\omega.
\end{align*}
Lemma is proved.
\end{proof}

A word that does not contain the symbol $\omega$ is called a \emph{semigroup word}. The \emph{last occurrence sequence} of a word $\mathbf w$, denoted by $\los(\mathbf w)$, is the semigroup word obtained from $\mathbf w$ by retaining only the last occurrence of each letter. If $\mathbf w$ does not contain letters (i.e., if $\mathbf w=\omega^n$ for some natural number $n$) then $\los(\mathbf w)$ is the empty word. The length of a semigroup word $\mathbf w$ is denoted by $\ell(\mathbf w)$. If a word $\mathbf w$ contains the symbol $\omega$ then we put $\ell(\mathbf w)=\infty$. We denote by $\con(\mathbf w)$ the \emph{content} of the word $\mathbf w$, i.e. the set of all letters occurring in $\mathbf w$.

\begin{corollary}
\label{cor w=los(w)omega}
If $\mathbf w$ is a word of a length $\ge 3$ then the variety $\mathbf{IS}$ satisfies the identity
\begin{equation}
\label{w=los(w)omega}
\mathbf w\approx\los(\mathbf w)\omega.
\end{equation}
\end{corollary}

\begin{proof}
Since $\ell(\mathbf w)\ge 3$, the identity \eqref{xyz=xyz omega} allows us to assume that $\mathbf w$ is not a semigroup word. Further, in view of the identity \eqref{omega x=x omega}, we can assume that $\mathbf{w=w'}\omega$ for some word $\mathbf w'$. The identities \eqref{xyx=yx omega} and \eqref{xxy=xy omega} permit to delete from $\mathbf w$ all but last occurrences of all letters. Therefore, we can assume that every letter from $\con(\mathbf w)$ occurs in $\mathbf w$ at most one time. Finally, the identities \eqref{omega omega=omega} and \eqref{omega x=x omega} allow us to delete from $\mathbf w$ all but last occurrences of the symbol $\omega$. Therefore, the identity \eqref{w=los(w)omega} holds in $\mathbf{IS}$.
\end{proof}

Note that the identities \eqref{omega omega=omega}--\eqref{xyz=xyz omega}, \eqref{xyx=yx omega} and \eqref{xxy=xy omega} are particular instances of the identity \eqref{w=los(w)omega}.

\begin{corollary}
\label{los(u)=los(v)}
If $\mathbf u$ and $\mathbf v$ are words such that $\ell(\mathbf u),\ell(\mathbf v)\ge 3$ and $\los(\mathbf u)=\los(\mathbf v)$ then the variety $\mathbf{IS}$ satisfies the identity $\mathbf{u\approx v}$.
\end{corollary}

\begin{proof}
Indeed, we have $\mathbf u\stackrel{\eqref{w=los(w)omega}}\approx\los(\mathbf u)\omega=\los(\mathbf v)\omega\stackrel{\eqref{w=los(w)omega}}\approx\mathbf v$.
\end{proof}

We need a description of the identities of a few concrete varieties of implication semigroups. We say that a word $\mathbf w$ \emph{contains a square} if $\mathbf{w=ab}^2\mathbf c$ for some word $\mathbf b$ and some (possibly empty) words $\mathbf a$ and $\mathbf c$.

\begin{lemma}
\label{word problem}
A non-trivial identity $\mathbf{ u\approx v}$ holds in the variety:
\begin{itemize}
\item[\textup{(i)}] $\mathbf{SL}$ if and only if $\con(\mathbf u)=\con(\mathbf v)$;
\item[\textup{(ii)}] $\mathbf B$ if and only if $\los(\mathbf u)=\los(\mathbf v)$;
\item[\textup{(iii)}] $\mathbf{ZM}$ if and only if $\ell(\mathbf u),\ell(\mathbf v)\ge 2$;
\item[\textup{(iv)}] $\mathbf K$ if and only if either $\mathbf{u\approx v}$ is the commutative law or each of the words $\mathbf u$ and $\mathbf v$ either contains a square or has a length $\ge 3$;
\item[\textup{(v)}] $\mathbf L$ if and only if each of the words $\mathbf u$ and $\mathbf v$ either contains a square or has a length $\ge 3$;
\item[\textup{(vi)}] $\mathbf M$ if and only if either $\mathbf{u\approx v}$ is the commutative law or $\ell(\mathbf u),\ell(\mathbf v)\ge 3$;
\item[\textup{(vii)}] $\mathbf N$ if and only if $\ell(\mathbf u),\ell(\mathbf v)\ge 3$.
\end{itemize}
\end{lemma}

\begin{proof}
(i) Lemma \ref{bands=monoids} implies that any identity that holds in the variety $\mathbf{SL}$ is equivalent to a semigroup identity. It is well known and can be easily checked that a semigroup identity $\mathbf{u\approx v}$ holds in the variety of semilattices $\mathbf{SL}$ if and only if $\con(\mathbf u)= \con(\mathbf v)$. This completes the proof of the assertion (i).

\smallskip

(ii) Lemma \ref{bands=monoids} shows that an arbitrary identity $\mathbf{u\approx v}$ is equivalent within $\mathbf B$ to the identity $\mathbf{u\omega\approx v\omega}$. This allows us to consider only identities, both sides of which have a length $\ge 3$. Corollary \ref{cor w=los(w)omega} allows us now to delete all but last occurrences of any letter in any word. Furthermore, Lemma \ref{bands=monoids} means that we can delete all occurrences of the symbol $\omega$ in every word. This proves the claim (ii).

\smallskip

(iii)--(vii) These assertions are evident.
\end{proof}

\begin{lemma}
\label{nsupseteq B,K,N}
Let $\mathbf V$ be a variety of implication semigroups.
\begin{itemize}
\item[\textup{(i)}] If $\mathbf{V\nsupseteq B}$ then $\mathbf V$ satisfies the identity
\begin{equation}
\label{xy omega=yx omega}
xy\omega\approx yx\omega.
\end{equation}
\item[\textup{(ii)}] If $\mathbf{V\nsupseteq K}$ then $\mathbf V$ satisfies the identity
\begin{equation}
\label{xy=xy omega}
xy\approx xy\omega.
\end{equation}
\item[\textup{(iii)}] If $\mathbf{V\nsupseteq N}$ then $\mathbf V$ satisfies either the commutative law or the identity
\begin{equation}
\label{x omega=xx}
x\omega\approx x^2.
\end{equation}
\end{itemize}
\end{lemma}

\begin{proof}
Lemma \ref{word problem}(i) implies that the identities \eqref{xy omega=yx omega}--\eqref{x omega=xx} hold in the variety $\mathbf{SL}$. Whence, any of these identities is valid in a variety $\mathbf V$ if and only if it is valid in $\mathbf{V\vee SL}$. Further, if $\mathbf{V\vee SL\nsupseteq X}$ where $\mathbf X$ is one of the varieties $\mathbf B$, $\mathbf K$ or $\mathbf N$ then $\mathbf{V\nsupseteq X}$. These observations allow us to assume that $\mathbf{V\supseteq SL}$. Then Lemma \ref{word problem}(i) applies and we conclude that if $\mathbf V$ satisfies an identity $\mathbf{u\approx v}$ then
\begin{equation}
\label{con(u)=con(v)}
\con(\mathbf u)=\con(\mathbf v).
\end{equation}

\smallskip

(i) By the hypothesis, there is an identity $\mathbf{u\approx v}$ that is true in $\mathbf V$ but fails in $\mathbf B$. Lemma \ref{word problem}(ii) implies that $\los(\mathbf u)\ne\los(\mathbf v)$. If one of the words $\mathbf u$ or $\mathbf v$ has a length $\le 2$ then we multiply the identity $\mathbf{u\approx v}$ by $x^2$ for some $x\in\con(\mathbf u)$ from the left. Clearly, this does not change the words $\los(\mathbf u)$ and $\los(\mathbf v)$. Both the sides of the resulting identity have a length $\ge 3$. Thus we can assume without loss of generality that $\ell(\mathbf u),\ell(\mathbf v)\ge 3$. According to Corollary \ref{cor w=los(w)omega}, we can assume also that $\mathbf u=\los(\mathbf u)\omega$ and $\mathbf v=\los(\mathbf v)\omega$. Then the equality \eqref{con(u)=con(v)} implies that $\con\left(\los(\mathbf u)\right)=\con\left(\los(\mathbf v)\right)$. Whence there are letters $x$ and $y$ such that $x$ precedes $y$ in $\mathbf u$ but $y$ precedes $x$ in $\mathbf v$. Now we substitute $\omega$ for all letters from $\con(\mathbf u)$ except $x$ and $y$ in the identity $\mathbf{u\approx v}$. In view of the identities \eqref{omega omega=omega} and \eqref{omega x=x omega}, the obtained identity implies \eqref{xy omega=yx omega}.

\smallskip

(ii) By the hypothesis, there is an identity $\mathbf{u\approx v}$ that is true in $\mathbf V$ but fails in $\mathbf K$. Lemma \ref{word problem}(iv) implies that the identity $\mathbf{u\approx v}$ is not the commutative law and one of the words $\mathbf u$ and $\mathbf v$, say $\mathbf u$, has a length $\le 2$ and is not a square. If $\ell(\mathbf u)=1$ then we can multiply the identity $\mathbf{u\approx v}$ by some letter $y\notin\con(\mathbf u)$ from the right. Thus we can assume that $\ell(\mathbf u)=2$, whence $\mathbf u=xy$. Further, $\mathbf v\ne yx$ because the identity $\mathbf{u\approx v}$ is not the commutative law. The equality \eqref{con(u)=con(v)} implies now that $\ell(\mathbf v)\ge 3$. Hence $\mathbf V$ satisfies the identities $xy=\mathbf{u\approx v}\stackrel{\eqref{w=los(w)omega}}\approx\los(\mathbf v)\omega$ and therefore, the identity
\begin{equation}
\label{xy=los(v)omega}
xy\approx\los(\mathbf v)\omega.
\end{equation}
If $\los(\mathbf v)=xy$ then we are done. Otherwise, $\los(\mathbf v)=yx$ and we have
$$
xy\stackrel{\eqref{xy=los(v)omega}}\approx yx\omega\stackrel{\eqref{xy=los(v)omega}}\approx xy\omega^2\stackrel{\eqref{omega omega=omega}}\approx xy\omega.
$$

\smallskip

(iii) By the hypothesis, there is an identity $\mathbf{u\approx v}$ that is true in $\mathbf V$ but fails in $\mathbf N$. Lemma \ref{word problem}(vii) implies that one of the words $\mathbf u$ and $\mathbf v$, say $\mathbf u$, has a length $\le 2$. We can assume without loss of generality that $\ell(\mathbf u)\le\ell(\mathbf v)$. If $\ell(\mathbf u)=1$ then we can multiply the identity $\mathbf{u\approx v}$ by some letter $y$ from the right. Thus we can assume that $\ell(\mathbf u)=2$, whence $\mathbf u\in\{xy,x^2\}$. Recall that the equality \eqref{con(u)=con(v)} holds.

Suppose at first that $\mathbf u=x^2$. Then $\ell(\mathbf v)\ge 3$ and $\los(\mathbf v)=x$. Therefore, $\mathbf V$ satisfies the identities $x^2=\mathbf{u\approx v}\stackrel{\eqref{w=los(w)omega}}\approx\los(\mathbf v)\omega=x\omega$. Whence, the identity \eqref{x omega=xx} holds in $\mathbf V$.

Finally, let $\mathbf u=xy$. If $\mathbf v=yx$ then the variety $\mathbf V$ is commutative, and we are done. Let now $\mathbf v\ne yx$. Then $\ell(\mathbf v)\ge 3$. Substituting $x$ for $y$ in $\mathbf{u\approx v}$, we get the situation considered in the previous paragraph.
\end{proof}

\begin{lemma}
\label{nsupseteq SL,ZM}
Let $\mathbf V$ be a variety of implication semigroups.
\begin{itemize}
\item[\textup{(i)}] If $\mathbf{V\nsupseteq SL}$ then $\mathbf{V\subseteq N}$.
\item[\textup{(ii)}] If $\mathbf{V\nsupseteq ZM}$ then $\mathbf{V\subseteq B}$.
\end{itemize}
\end{lemma}

\begin{proof}
(i) Suppose that $\mathbf{SL\nsubseteq V}$. In view of Lemma \ref{word problem}(i), $\mathbf V$ satisfies an identity $\mathbf{u\approx v}$ with $\con(\mathbf u)\ne\con(\mathbf v)$. We can assume without loss of generality that there is a letter $x\in\con(\mathbf u)\setminus\con(\mathbf v)$. One can substitute $\omega$ to all letters except $x$ in the identity $\mathbf{u\approx v}$. The identities \eqref{omega omega=omega} and \eqref{omega x=x omega} imply that $\mathbf V$ satisfies the identity $x^k\omega\approx\omega$ for some $k$. Corollary \ref{los(u)=los(v)} implies that the variety $\mathbf{IS}$ satisfies the identity $xyz\approx(xyz)^n\omega$ for any natural $n$. Then the identities $xyz\approx(xyz)^k\omega\approx\omega$ hold in $\mathbf V$, whence $\mathbf{V\subseteq N}$.

\smallskip

(ii) Suppose that $\mathbf{ZM\nsubseteq V}$. It suffices to verify that $\mathbf{V\vee SL\subseteq B}$. Whence, we can assume that $\mathbf{V\supseteq SL}$. In view of Lemma \ref{word problem}(iii), $\mathbf V$ satisfies a non-trivial identity of the kind $x\approx\mathbf v$. Lemma \ref{word problem}(i) implies that $\con(\mathbf v)=\{x\}$. Clearly, $\ell(\mathbf v)>1$. If $\ell(\mathbf v)=2$ then $\mathbf v=x^2$, whence $\mathbf{V\subseteq B}$. Let now $\ell(\mathbf v)\ge 3$. Then Corollary \ref{los(u)=los(v)} applies and we conclude that the identity $\mathbf v^2\approx\mathbf v$ is true in the variety $\mathbf{IS}$. Therefore, $\mathbf V$ satisfies the identities $x^2\approx\mathbf v^2\approx\mathbf v\approx x$, whence $\mathbf{V\subseteq B}$ again.
\end{proof}

\subsection{The structure of the lattices $L(\mathbf B)$ and $L(\mathbf N)$}
\label{L(B) and L(N)}

According to Lemma \ref{bands=monoids}, $\mathbf B$ is a monoid variety. Therefore, it satisfies the identities $xyx\stackrel{\eqref{xyx=yx omega}}\approx yx\omega\approx yx$. The lattice of varieties of band monoids is completely described in \cite{Wismath-86}. In view of \cite[Proposition 4.7]{Wismath-86}, the lattice $L(\mathbf B)$ is the 3-element chain $\mathbf{T\subset SL\subset B}$.

The variety $\mathbf N$ satisfies the identities $\omega x\stackrel{\eqref{omega x=x omega}}\approx x\omega\stackrel{\eqref{omega omega=omega}}\approx x\omega^2\approx\omega$. Hence every semigroup from $\mathbf N$ contains the zero element and the operation $\omega$ fixes just this element in each semigroup from $\mathbf N$. This means that $\mathbf N$ is nothing but the variety of all nilpotent of degree 3 semigroups. The subvariety lattice of this variety has the form shown in Fig. \ref{IS}. This claim can be easily verified directly and is a part of a semigroup folklore. It is known at least from the beginning of 1970's (see \cite{Melnik-72}, for instance).

\subsection{Identity bases for certain varieties}
\label{five identity bases}

Here we prove the following equalities:
\begin{align}
\label{SL vee K}&\mathbf{SL\vee K}=\var\{x\omega\approx x^2,\,xy\approx yx\},\\
\label{SL vee L}&\mathbf{SL\vee L}=\var\{x\omega\approx x^2,\,xy\omega\approx yx\omega\},\\
\label{SL vee M}&\mathbf{SL\vee M}=\var\{xy\approx yx\},\\
\label{SL vee N}&\mathbf{SL\vee N}=\var\{xy\omega\approx yx\omega\},\\
\label{SL vee ZM}&\mathbf{SL\vee ZM}=\var\{xy\approx yx\omega\}.
\end{align}
Let $\varepsilon$ be one of these equalities. The fact that the left-hand side of $\varepsilon$ is contained in its right-hand side is evident. One can verify reverse inclusions.

We start with the equality \eqref{SL vee N}. Let $\mathbf{u\approx v}$ be an identity that holds in the variety $\mathbf{SL\vee N}$. The items (i) and (vii) of Lemma \ref{word problem} imply that $\con(\mathbf u)=\con(\mathbf v)$ and $\ell(\mathbf u),\ell(\mathbf v)\ge 3$. Put $\con(\mathbf u)=\con(\mathbf v)=\{x_1,x_2,\dots,x_n\}$. Then Corollary \ref{cor w=los(w)omega} allows us to assume that our identity coincides with the identity
$$
x_1x_2\cdots x_n\omega\approx x_{\pi(1)}x_{\pi(2)}\cdots x_{\pi(n)}\omega
$$
for some permutation $\pi$ on the set $\{1,2,\dots,n\}$. But an arbitrary identity of such a form follows from the identities \eqref{omega omega=omega}, \eqref{omega x=x omega} and \eqref{xy omega=yx omega}, whence $\var\{xy\omega\approx yx\omega\}\subseteq\mathbf{SL\vee N}$. This proves the equality \eqref{SL vee N}.

The items (vi) and (vii) of Lemma \ref{word problem} imply that a unique identity that holds in $\mathbf{SL}\vee\mathbf M$ but fails in $\mathbf{SL}\vee\mathbf N$ is the commutative law. In view of what was said in the previous paragraph, this implies the equality \eqref{SL vee M}.

Let now $\mathbf{u\approx v}$ be an arbitrary identity that holds in $\mathbf{SL\vee L}$. The items (i) and (v) of Lemma \ref{word problem} imply that $\con(\mathbf u)=\con(\mathbf v)$ and each of the words $\mathbf u$ and $\mathbf v$ either has a length $\ge 3$ or contains a square. If one of the sides of the identity $\mathbf{u\approx v}$ contains a square then we can apply the identity \eqref{x omega=xx}. Both the sides of the obtained identity will have a length $\ge 3$. Thus, we can assume that $\ell(\mathbf u),\ell(\mathbf v)\ge 3$. Then the items (i) and (vii) of Lemma \ref{word problem} apply with the conclusion that the identity $\mathbf{u\approx v}$ holds in $\mathbf{SL\vee N}$. The equality \eqref{SL vee N} implies that $\var\{x\omega\approx x^2,\,xy\omega\approx yx\omega\}\subseteq\mathbf{SL\vee N}$. Therefore, the identity $\mathbf{u\approx v}$ holds in $\var\{x\omega\approx x^2,\,xy\omega\approx yx\omega\}$. Thus, the last variety is contained in $\mathbf{SL\vee L}$. This proves the equality \eqref{SL vee L}.

Let now $\mathbf{u\approx v}$ be an arbitrary identity that holds in $\mathbf{SL\vee K}$. Suppose that this identity differs from the commutative law. Comparison of the items (iv) and (v) of Lemma \ref{word problem} shows that this identity holds in $\mathbf{SL\vee L}$. Besides that, the variety $\mathbf{SL\vee K}$ is commutative. The equality \eqref{SL vee L} proved in the previous paragraph implies now that the identity $\mathbf{u\approx v}$ holds in the variety
$$
\var\{x\omega\approx x^2,\,xy\omega\approx yx\omega,\,xy\approx yx\}=\var\{x\omega\approx x^2,\,xy\approx yx\}.
$$
Thus, the last variety is contained in $\mathbf{SL\vee K}$ in any case. This proves the equality \eqref{SL vee K}.

Finally, let $\mathbf{u\approx v}$ be an identity that holds in $\mathbf{SL\vee ZM}$. The items (i) and (iii) of Lemma \ref{word problem} imply that $\con(\mathbf u)=\con(\mathbf v)$ and $\ell(\mathbf u),\ell(\mathbf v)\ge 2$. Further, the items (i) and (vii) of the same lemma imply that the identity
\begin{equation}
\label{u omega=v omega}
\mathbf u\omega\approx\mathbf v\omega
\end{equation}
holds in $\mathbf{SL\vee N}$. The equality \eqref{SL vee N} implies that the last identity follows from the identity \eqref{xy omega=yx omega}. In turn, the last identity follows from the identity
\begin{equation}
\label{xy=yx omega}
xy\approx yx\omega.
\end{equation}
Indeed, if we multiply \eqref{xy=yx omega} on $\omega$ from the right and use the identity \eqref{omega omega=omega}, we obtain \eqref{xy omega=yx omega}. Summarizing all we say, we get that the identity \eqref{u omega=v omega} holds in the variety $\var\{xy\approx yx\omega\}$. Besides that, it is clear that $\var\{xy\approx yx\omega\}\nsupseteq\mathbf K$. Lemma \ref{nsupseteq B,K,N}(ii) implies that the variety $\var\{xy\approx yx\omega\}$ satisfies the identity \eqref{xy=xy omega}. Hence this variety satisfies the identities $\mathbf u\stackrel{\eqref{xy=xy omega}}\approx\mathbf u\omega\stackrel{\eqref{u omega=v omega}}\approx\mathbf v\omega\stackrel{\eqref{xy=xy omega}}\approx\mathbf v$. The equality \eqref{SL vee ZM} is proved.

\subsection{The structure of the lattice $L(\mathbf{SL\vee ZM})$}
\label{L(SL vee ZM)}

Here we aim to verify that the variety $\mathbf{SL\vee ZM}$ contains only four subvarieties, namely $\mathbf T$, $\mathbf{SL}$, $\mathbf{ZM}$ and $\mathbf{SL\vee ZM}$. Let $\mathbf{V\subset SL\vee ZM}$. Results of Subsection \ref{L(B) and L(N)} imply that $\mathbf{SL}$ and $\mathbf{ZM}$ are minimal non-trivial varieties of implication semigroups. Thus, it suffices to check that $\mathbf V$ is contained in one of the varieties $\mathbf{SL}$ or $\mathbf{ZM}$. Clearly, either $\mathbf{SL\nsubseteq V}$ or $\mathbf{ZM\nsubseteq V}$. If $\mathbf{ZM\nsubseteq V}$ then Lemma \ref{nsupseteq SL,ZM}(ii) implies that $\mathbf{V\subseteq B}$. Since the variety $\mathbf{SL\vee ZM}$ and therefore, $\mathbf V$ is commutative, we have that $\mathbf{V\subseteq SL}$ in this case. Suppose now that $\mathbf{SL\nsubseteq V}$. Lemma \ref{nsupseteq SL,ZM}(i) implies that $\mathbf{V\subseteq N}$. Suppose that $\mathbf{K\subseteq V}$. Then $\mathbf{K\subseteq SL\vee ZM}$. The equality \eqref{SL vee ZM} implies then that the identity \eqref{xy=yx omega} is true in $\mathbf K$. But this identity implies within $\mathbf N$ the identities $xy\approx yx\omega\approx\omega$ (see the second paragraph in Subsection \ref{L(B) and L(N)}). The identity $xy\approx\omega$ evidently fails in $\mathbf K$, a contradiction. We prove that $\mathbf{V\subseteq N}$ but $\mathbf{K\nsubseteq V}$. The description of the lattice $L(\mathbf N)$ given in Subsection \ref{L(B) and L(N)} implies now that $\mathbf{V\subseteq ZM}$ (see Fig. \ref{IS}).

\subsection{The structure of the lattice $L(\mathbf{SL\vee N})$}
\label{L(SL vee N)}

Here we are going to prove that the lattice $L(\mathbf{SL\vee N})$ has the form shown in Fig. \ref{IS}. Let $\mathbf{V\subset SL\vee N}$. By Lemma \ref{nsupseteq SL,ZM}(i), the lattice $L(\mathbf{SL\vee N})$ is the set-theoretical union of the lattice $L(\mathbf N)$ and the interval $[\mathbf{SL},\mathbf{SL\vee N}]$. The lattice $L(\mathbf N)$ is described in Subsection \ref{L(B) and L(N)}. It remains to verify that the interval $[\mathbf{SL},\mathbf{SL\vee N}]$ has the form shown in Fig. \ref{IS}. In other words, we need to check that if $\mathbf{SL\subseteq V\subset SL\vee N}$ then $\mathbf V$ coincides with one of the varieties $\mathbf{SL}$, $\mathbf{SL\vee ZM}$, $\mathbf{SL\vee K}$, $\mathbf{SL\vee L}$ or $\mathbf{SL\vee M}$. Since $\mathbf{N\nsubseteq V}$, Lemma \ref{nsupseteq B,K,N}(iii) and the equalities \eqref{SL vee L} and \eqref{SL vee M} imply that $\mathbf V$ is contained in one of the varieties $\mathbf{SL\vee L}$ or $\mathbf{SL\vee M}$. If $\mathbf V$ coincides with one of these varieties then we are done. Let now $\mathbf{V\ne SL\vee L}$ and $\mathbf{V\ne SL\vee M}$.

Suppose that $\mathbf{SL\vee K\subseteq V}$. If $\mathbf{V\subset SL\vee L}$ then the variety $\mathbf V$ is commutative by the items (iv) and (v) of Lemma \ref{word problem}. Then the equalities \eqref{SL vee K} and \eqref{SL vee L} show that $\mathbf{V=SL\vee K}$. Suppose now that $\mathbf{V\subset SL\vee M}$. Then the items (iv) and (vi) of Lemma \ref{word problem} imply that $\mathbf V$ satisfies a non-trivial identity of the kind $\mathbf{u\approx v}$ where the word $\mathbf u$ contains a square and $\ell(\mathbf u)<3$. Clearly, $\mathbf u=x^2$. Thus, $\mathbf V$ satisfies a non-trivial identity of the kind $x^2\approx\mathbf v$. Lemma \ref{word problem}(i) implies that $\con(\mathbf v)=\{x\}$. It is clear that $\ell(\mathbf v)\ne 2$. If $\ell(\mathbf v)=1$ then our identity implies $x^3\approx x^2$. Finally, let $\ell(\mathbf v)\ge 3$. Then $\mathbf V$ satisfies the identities $x^2\approx\mathbf v\stackrel{\eqref{w=los(w)omega}}\approx\los(\mathbf v)\omega=x\omega$. Besides that, $\mathbf V$ is commutative because $\mathbf{V\subseteq SL\vee M}$. The equality \eqref{SL vee K} implies now that $\mathbf{V=SL\vee K}$.

It remains to consider the case when $\mathbf{SL\vee K\nsubseteq V}$. Lemma \ref{nsupseteq B,K,N}(ii) and the equality \eqref{SL vee N} imply that the identities $xy\stackrel{\eqref{xy=xy omega}}\approx xy\omega\stackrel{\eqref{xy omega=yx omega}}\approx yx\omega$ hold in $\mathbf V$. Therefore, $\mathbf V$ satisfies the identity \eqref{xy=yx omega}. This fact together with the equality \eqref{SL vee ZM} imply that $\mathbf{V\subseteq SL\vee ZM}$. The result of Subsection \ref{L(SL vee ZM)} implies now that either $\mathbf{V=SL\vee ZM}$ or $\mathbf{V=SL}$ (see Fig. \ref{IS}).

\subsection{The structure of the lattice $L(\mathbf{B\vee ZM})$}
\label{L(B vee ZM)}

Here we aim to check that the lattice $L(\mathbf{B\vee ZM})$ has the form shown in Fig. \ref{IS}. First of all we prove that
\begin{equation}
\label{B vee ZM}
\mathbf{B\vee ZM}=\var\{xy\approx xy\omega\}.
\end{equation}
The inclusion $\mathbf{B\vee ZM}\subseteq\var\{xy\approx xy\omega\}$ follows from Lemma \ref{bands=monoids} and the definition of the variety $\mathbf{ZM}$. One can verify the inverse inclusion. Let $\mathbf{u\approx v}$ be an arbitrary identity that holds in $\mathbf{B\vee ZM}$. The items (ii) and (iii) of Lemma \ref{word problem} imply that $\los(\mathbf u)=\los(\mathbf v)$ and $\ell(\mathbf u),\ell(\mathbf v)\ge 2$. Clearly, $\los(\mathbf u\omega)=\los(\mathbf v\omega)$. Now Corollary \ref{los(u)=los(v)} implies that the variety $\mathbf{IS}$ satisfies the identity $\mathbf u\omega\approx\mathbf v\omega$. Then the variety $\var\{xy\approx xy\omega\}$ satisfies the identities $\mathbf u\stackrel{\eqref{xy=xy omega}}\approx\mathbf u\omega\approx\mathbf v\omega\stackrel{\eqref{xy=xy omega}}\approx\mathbf v$, whence $\mathbf{u\approx v}$ holds in $\var\{xy\approx xy\omega\}$. The equality \eqref{B vee ZM} is proved.

Let $\mathbf{V\subset B\vee ZM}$. In view of the results of Subsections \ref{L(B) and L(N)} and \ref{L(SL vee ZM)}, it remains to verify that either $\mathbf{ V\subseteq B}$ or $\mathbf{V\subseteq SL\vee ZM}$. Clearly, either $\mathbf{ZM\nsubseteq V}$ or $\mathbf{B\nsubseteq V}$. If $\mathbf{ZM\nsubseteq V}$ then Lemma \ref{nsupseteq SL,ZM}(ii) implies that $\mathbf V\subseteq\mathbf B$. Suppose now that $\mathbf{B\nsubseteq V}$. Then Lemma \ref{nsupseteq B,K,N}(i) implies that $\mathbf V$ satisfies the identity \eqref{xy omega=yx omega}. Besides that, the equality \eqref{B vee ZM} implies that the identity \eqref{xy=xy omega} holds in $\mathbf V$. Therefore, the identities $xy\stackrel{\eqref{xy=xy omega}}\approx xy\omega\stackrel{\eqref{xy omega=yx omega}}\approx yx\omega$ are satisfied by $\mathbf V$. We see that $\mathbf V$ satisfies the identity \eqref{xy=yx omega}. This fact and the equality \eqref{SL vee ZM} implies that $\mathbf{V\subseteq SL\vee ZM}$.

\subsection{Completion of the proof}
\label{end of the proof}

First of all, we verify that $\mathbf{IS=B\vee N}$. Let $\mathbf{u\approx v}$ be an arbitrary identity that holds in $\mathbf{B\vee N}$. The items (ii) and (vii) of Lemma \ref{word problem} imply that $\los(\mathbf u)=\los(\mathbf v)$ and $\ell(\mathbf u),\ell(\mathbf v)\ge 3$. Then Corollary \ref{los(u)=los(v)} implies that the variety $\mathbf{IS}$ satisfies the identity $\mathbf{u\approx v}$, and we are done.

Now we verify that
\begin{equation}
\label{B vee K}
\mathbf B\vee\mathbf K=\var\{x\omega\approx x^2\}.
\end{equation}
Lemma \ref{bands=monoids} and the definition of the variety $\mathbf B$ imply that this variety satisfies the identity \eqref{x omega=xx}. The variety $\mathbf K$ also satisfies this identity by Lemma \ref{word problem}(iv). Hence $\mathbf{B\vee K}\subseteq\var\{x\omega\approx x^2\}$. One can verify the inverse inclusion. Let $\mathbf{u\approx v}$ be an arbitrary identity that holds in $\mathbf{B\vee K}$. The items (ii) and (iv) of Lemma \ref{word problem} and the fact that the variety $\mathbf B$ is non-commutative imply that $\los(\mathbf u)=\los(\mathbf v)$ and each of the words $\mathbf u$ and $\mathbf v$ either has a length $\ge 3$ or contains a square. If a word contains a square then we can apply the identity \eqref{x omega=xx} and obtain a word of a length $\ge 3$. Thus, we can assume that $\ell(\mathbf u),\ell(\mathbf v)\ge 3$. Then Corollary \ref{los(u)=los(v)} implies that the identity $\mathbf{u\approx v}$ holds in the variety $\mathbf{IS}$ and therefore, in $\var\{x\omega\approx x^2\}$. The equality \eqref{B vee K} is proved.

Now we are well prepared to quickly complete the proof of Theorem \ref{main result}. Let $\mathbf{V\subset IS}$. Since $\mathbf{IS=B\vee N}$, either $\mathbf{B\nsubseteq V}$ or $\mathbf {N\nsubseteq V}$. In the former case, Lemma \ref{nsupseteq B,K,N}(i) and the equality \eqref{SL vee N} imply that $\mathbf{V\subseteq SL\vee N}$, while in the latter case $\mathbf{V\subseteq B\vee K}$ by Lemma \ref{word problem}(iv) and the equality \eqref{B vee K}. As we have seen in Subsection \ref{L(SL vee N)}, the lattice $L(\mathbf{SL\vee N})$ has the form shown in Fig. \ref{IS}. It remains to consider the lattice $L(\mathbf{B\vee K})$.

Let $\mathbf{V\subset B\vee K}$. Then either $\mathbf{B\nsubseteq V}$ or $\mathbf{K\nsubseteq V}$. As we have seen in the previous paragraph, $\mathbf{V\subseteq SL\vee N}$ in the former case. Then $\mathbf{V\subseteq(SL\vee N)\wedge(B\vee K)}$. Now the equalities \eqref{SL vee L}, \eqref{SL vee N} and \eqref{B vee K} apply with the conclusion that $\mathbf{V\subseteq SL\vee L}$. According to results of Subsection \ref{L(SL vee N)}, the lattice $L(\mathbf{SL\vee L})$ has the form shown in Fig. \ref{IS}. Finally, let $\mathbf{K\nsubseteq V}$. Then Lemma \ref{nsupseteq B,K,N}(ii) and the equality \eqref{B vee ZM} imply that $\mathbf{V\subseteq B\vee ZM}$. It remains to refer to results of Subsection \ref{L(B vee ZM)}.

Theorem \ref{main result} is proved.\qed 

\section*{APPENDIX B. On Problem \ref{0-distributive?}}

\refstepcounter{section}
\label{about 0-distr}

Here we consider arbitrary varieties of implication zroupoids without assumption that the binary operation is associative. For this reason, we return below to the original notation of the binary and \mbox{0-ary} operations by $\to$ and 0 respectively. 

As we have already noted in Section \ref{problems}, Problem \ref{0-distributive?} is equivalent to the following claim: if $\mathbf A$ is one of the varieties $\mathbf{ZM}$, $\mathbf{SL}$ or $\mathbf{BA}$, while $\mathbf X$, $\mathbf Y$ are varieties of implication zroupoids with $\mathbf{X,Y\nsupseteq A}$ then $\mathbf{X\vee Y\nsupseteq A}$. Here we prove the following

\begin{proposition}
\label{without SL or BA}
If $\mathbf A$ is one of the varieties $\mathbf{SL}$ or $\mathbf{BA}$, while $\mathbf X$, $\mathbf Y$ are varieties of implication zroupoids with $\mathbf{X,Y\nsupseteq A}$ then $\mathbf{X\vee Y\nsupseteq A}$.
\end{proposition}

\begin{proof}
First of all, we consider the following two 2-element implication zroupoids $\mathbf{2_s}$ and $\mathbf{2_b}$:

\begin{center}
$\mathbf{2_s}$:\enskip 
\begin{tabular}{r|rr}
$\to$&0&1\\
\hline
0&0&1\\
1&1&1
\end{tabular}
\hskip 2cm
$\mathbf{2_b}$:\enskip 
\begin{tabular}{r|rr}
$\to$&0&1\\
\hline
0&1&1\\
1&0&1
\end{tabular}
\end{center}

\noindent It is evident that the algebras $\mathbf{2_s}$ and $\mathbf{2_b}$ generate the varieties $\mathbf{SL}$ and $\mathbf{BA}$ respectively.

First, we shall prove Proposition \ref{without SL or BA} with $\mathbf{A=SL}$. In fact, this claim is well known and can be easily deduced from the known universal-algebraic facts summarized in the survey \cite{Plonka-Romanowska-92}, for instance. However, for the sake of completeness, we prefer to provide here a simple proof not depending on any other results.

Let $\mathbf X$ and $\mathbf Y$ be varieties of implication zroupoids with $\mathbf{X,Y\nsupseteq SL}$. Lemma \ref{word problem}(i) implies that the variety $\mathbf X$ satisfies some identity $\mathbf{a\approx b}$ such that $\con(\mathbf a)\ne\con(\mathbf b)$. We can assume without less of generality that there is a variable $x$ with $x\in\con(\mathbf a)\setminus\con(\mathbf b)$. One can substitute 0 for all variables occurring in the identity $\mathbf{a\approx b}$ except $x$. Then we obtain an identity of the form $\mathbf u(x)\approx\mathbf p$ such that $\con(\mathbf u)=\{x\}$ and $\con(\mathbf p)=\varnothing$. Clearly, this identity holds in $\mathbf X$. Analogously, the variety $\mathbf Y$ satisfies an identity of the form $\mathbf v(x)\approx\mathbf q$ such that $\con(\mathbf v)=\{x\}$ and $\con(\mathbf q)=\varnothing$. It is evident that the variety $\mathbf X$ satisfies the identity
\begin{equation}
\label{identity in the join}
\mathbf v(\mathbf u(x))\approx\mathbf v(\mathbf p).
\end{equation}
Furthermore, substituting $\mathbf u(x)$ first, and then $\mathbf p$ for $x$ in $\mathbf v(x)\approx\mathbf q$, we get that the identities $\mathbf v(\mathbf u(x))\approx\mathbf q$ and $\mathbf v(\mathbf p)\approx\mathbf q$ are satisfied in $\mathbf Y$. This means that the identity \eqref{identity in the join} holds in $\mathbf Y$ and therefore, in $\mathbf{X\vee Y}$. Since $\con(\mathbf v(\mathbf u(x)))=\{x\}$ and $\con(\mathbf v(\mathbf p))=\varnothing$, we can apply Lemma \ref{word problem}(i) and conclude that $\mathbf{X\vee Y\nsupseteq SL}$. Thus, we have proved Proposition \ref{without SL or BA} with $\mathbf{A=SL}$. 

Next, to prove Proposition \ref{without SL or BA} with $\mathbf{A=SL}$. To prove this proposition with $\mathbf{A=BA}$, we need the following

\begin{lemma}
\label{general properties}
The variety $\mathbf{IZ}$ satisfies the following identities:
\begin{align}
\label{(((0'x)y)z =(xy)z}&(x\to y)\to z\approx((0'\to x)\to y)\to z,\\
\label{(xy)z=((xy)z)''}&(x\to y)\to z\approx((x\to y)\to z)'',\\
\label{((0x)0')y=(x0'))y}&((0\to x)\to 0')\to y\approx(x\to 0')\to y,\\
\label{0'0'=0'}&0'\to 0'\approx 0',\\ 
\label{00'=0'}&0\to 0'\approx 0'. 
\end{align}
\end{lemma}

\begin{proof}
The identities \eqref{(((0'x)y)z =(xy)z} and \eqref{(xy)z=((xy)z)''} were found in \cite[Lemma 7.5(d)]{Sankappanavar-12} and \cite[Lemma 2.8(2)]{Cornejo-Sankappanavar-17b} respectively. This is the deduction of the identity \eqref{((0x)0')y=(x0'))y}:
\begin{align*}
((0\to x)\to 0')\to y &\stackrel{\eqref{I}}\approx((0''\to 0)\to(x\to 0')')'\to y\\
&\stackrel{\eqref{0''=0}}\approx(0'\to(x\to 0')')'\to y\stackrel{\eqref{(((0'x)y)z =(xy)z}}\approx(x\to 0')''\to y\\
&\stackrel{\eqref{(((0'x)y)z =(xy)z}}\approx((0'\to x)\to 0')''\to y\stackrel{\eqref{(xy)z=((xy)z)''}}\approx((0'\to x)\to 0')\to y\\
&\stackrel{\eqref{(((0'x)y)z =(xy)z}}\approx(x\to 0')\to y,
\end{align*}
Now we have to verify the following weakened version of the identity \eqref{00'=0'}:
\begin{equation}
\label{(00')x=(0')x} (0\to 0') \to x\approx 0'\to x.
\end{equation}
Indeed, we have
\begin{align*}
(0\to 0')\to x&\stackrel{\eqref{((0x)0')y=(x0'))y}}\approx(0'\to 0')\to x\stackrel{\eqref{(((0'x)y)z =(xy)z}}\approx((0'\to 0')\to 0')\to x\\
&\stackrel{\eqref{(xy)z=((xy)z)''}}\approx((0'\to 0')\to 0')''\to x\stackrel{\eqref{(((0'x)y)z =(xy)z}}\approx(0'\to 0')''\to x\\
&\stackrel{\eqref{(((0'x)y)z =(xy)z}}\approx 0'''\to x \stackrel{\eqref{0''=0}}\approx 0'\to x.
\end{align*}
Now we are ready to deduct the identities \eqref{0'0'=0'} and \eqref{00'=0'}:
\begin{center}
\begin{tabular}{lll}
\eqref{0'0'=0'}:&$0'\to 0'$&$\stackrel{\eqref{(00')x=(0')x}}\approx(0\to 0')\to 0'\stackrel{\eqref{I}}\approx((0''\to 0)\to(0'\to 0')')' $\\
&&$\stackrel{\eqref{((0x)0')y=(x0'))y}}\approx((0''\to 0)\to(0\to 0')')'\stackrel{\eqref{(00')x=(0')x}}\approx((0''\to 0)\to 0'')' $\\
&&$\stackrel{\eqref{0''=0}}\approx((0\to 0)\to 0)'=0'''\stackrel{\eqref{0''=0}}\approx 0'$,\\[\smallskipamount]
\eqref{00'=0'}:&$0\to 0'$&$\stackrel{\eqref{0''=0}}\approx 0''\to 0'=(0'\to 0)\to 0'\stackrel{\eqref{I}}\approx((0''\to 0')\to(0\to 0')')'$\\
&&$\stackrel{\eqref{0''=0}}\approx((0\to 0')\to(0\to 0')')'\stackrel{\eqref{(00')x=(0')x}}\approx(0'\to(0\to 0')')'$\\
&&$\stackrel{\eqref{(00')x=(0')x}}\approx(0'\to 0'')'\stackrel{\eqref{0''=0}}\approx(0'\to 0)'=0'''\stackrel{\eqref{0''=0}}\approx 0'$.
\end{tabular}
\end{center}
Lemma is proved.
\end{proof}

Now we are ready to prove the following assertion that evidently implies Proposition \ref{without SL or BA} with $\mathbf{A=BA}$.

\begin{lemma}
\label{without BA}
A variety of implication zroupoids $\mathbf X$ does not contain the variety $\mathbf{BA}$ if and only if $\mathbf X$ satisfies the identity
\begin{equation}
\label{0 equals 0 prime}
0\approx 0'.
\end{equation}
\end{lemma}

\begin{proof}
The ``only if'' part is evident because $\mathbf{2_b}$ does not satisfy \eqref{0 equals 0 prime}. Now we verify the ``if'' part. Arguing a contradiction, we suppose that $\mathbf{X\nsupseteq BA}$ and $\mathbf X$ does not satisfy the identity \eqref{0 equals 0 prime}. Let $A$ be an algebra in $\mathbf X$ such that \eqref{0 equals 0 prime} wrong in $A$. Put $1:=0'$. Observe that
$$
0\to 1=0\to 0'\stackrel{\eqref{00'=0'}}=1\stackrel{\eqref{0'0'=0'}}=0'\to 0'=1\to 1\text{ and } 1\to 0=0'\to 0=0''\stackrel{\eqref{0''=0}}=0.
$$
Hence $\langle\{0,1\},\to,0\rangle$ is a subalgebra of $A$ which is isomorphic to $\mathbf{2_b}$, implying that $\mathbf{2_b}\in\mathbf X$.
\end{proof}

Thus, we have proved Proposition \ref{without SL or BA} with $\mathbf{A=BA}$. This completes the proof of this proposition as a whole.
\end{proof}

\subsection*{Acknowledgements}
This work was started when the second and the third authors took part in the Emil Artin International Conference held in Yerevan in May--June of 2018. The authors are deeply grateful to Professor Yuri Movsisyan and his colleagues for the excellent organization of the conference and the creation of the favorable atmosphere that contributed to the appearance of this article. The authors would like to express also their gratitude to the anonymous referee for his/her valuable remarks and suggestions that contributed to a significant improvement of the original version of the manuscript, in general and to the proof of Theorem \ref{main result} given in Section \ref{proof}, in particular.

\end{document}